\documentclass[12pt,twoside]{amsart}

\usepackage{amssymb}
\usepackage{graphicx}

\usepackage[margin=1in]{geometry}
\newcommand{\NN}{\mathbb{N}}
\newcommand{\PP}{\mathbb{P}}
\newcommand{\ZZ}{\mathbb{Z}}

\newcommand{\oI}{\overline{I}}
\newcommand{\oIt}{\overline{I_t}}

\DeclareMathOperator{\codim}{codim}
\DeclareMathOperator{\charf}{char}
\def\floor#1{\left\lfloor #1 \right\rfloor}

\newcommand{\triplestack}[3]{\begin{array}{c} #1 \\ #2 \\ #3 \end{array}}

\newcommand{\st}{\; | \;}
\numberwithin{figure}{section}
\numberwithin{equation}{section}

\newtheorem{theorem}[equation]{Theorem}
\newtheorem{lemma}[equation]{Lemma}
\newtheorem{proposition}[equation]{Proposition}
\newtheorem{corollary}[equation]{Corollary}
\newtheorem{conjecture}[equation]{Conjecture}

\theoremstyle{definition}

\newtheorem{remark}[equation]{Remark}
\newtheorem{example}[equation]{Example}

\def\today{\number\day.\space
           \ifcase\month\or January\or February\or March\or April\or May\or June\or July\or August\or September\or October\or November\or December\fi
           \space\number\year}

\begin{document}

\title{Hyperplane sections and the subtlety of the Lefschetz properties}
\author[D.\ Cook II, U.\ Nagel]{David Cook II, Uwe Nagel}
\address{Department of Mathematics, University of Kentucky, 715 Patterson Office Tower, Lexington, KY 40506-0027, USA}
\email{dcook@ms.uky.edu, uwe.nagel@uky.edu}
\thanks{Part of the work for this paper was done while the authors were partially supported by the National Security Agency under Grant Number H98230-09-1-0032.}

\begin{abstract}
    The weak and strong Lefschetz properties are two basic properties that Artinian algebras may have.  Both Lefschetz
    properties may vary under small perturbations or changes of the characteristic.   We study these subtleties by proposing
    a systematic way of deforming a monomial ideal failing the weak Lefschetz property to an ideal with the same Hilbert 
    function and  the weak Lefschetz property.  In particular, we lift a family of Artinian monomial ideals to  finite
    level sets of points in projective space with the property that a general hyperplane section has the weak Lefschetz 
    property in almost all characteristics, whereas a special hyperplane section does not have this property in any characteristic.
\end{abstract}

\maketitle

\section{Introduction} \label{sec:introduction}

Let $K$ be an infinite field and $I$ be a homogeneous Artinian ideal in $R = K[x_1, \ldots, x_n]$.  The ring $A = R/I$ is  {\em level of
type} $t$ if its socle has dimension $t$ and is concentrated in one degree. The algebra $A$ is said to have the {\em weak Lefschetz
property} if there is a linear form $\ell \in A$, called a {\em weak Lefschetz element} of $A$, such that for all integers $d$, the
multiplication map $\times \ell: [A]_d \rightarrow [A]_{d+1}$ has maximal rank, that is, it is surjective or injective. Further, $A$
is said to have the {\em strong Lefschetz property} if there is a linear form $\ell \in A$, called a {\em strong Lefschetz element} of
$A$, such that for all positive $k$ and all integers $d$ the multiplication map $\times\ell^k: [A]_d \rightarrow [A]_{d+k}$ has
maximal rank. Clearly the strong Lefschetz property implies the weak Lefschetz property.

Both Lefschetz properties have been studied extensively, especially for the constraints on the Hilbert function (see, e.g.,
\cite{BMMNZ}, \cite{HMNW}, \cite{MZ}, and~\cite{ZZ}).  For example, algebras with the weak Lefschetz property have strictly unimodal
Hilbert functions~\cite[Remark~3.3(4)]{HMNW}. Despite their utility---\cite{St} is a well known example---much is still unknown
about the presence of the Lefschetz properties, even in seemingly simple cases (see, e.g., \cite{BK-2010} and~\cite{MMN}).

The weak Lefschetz property may subtly vary under changes of the base field characteristic and  under deformation (see, e.g.,
\cite{MMN}). Even in the case of a monomial complete intersection in three variables it is not completely known in which positive
characteristics the weak Lefschetz property is present though  very interesting partial results have been established in~\cite{LZ}
and~\cite{BK-2010}. On the other hand, in~\cite[Section~5]{MMN} it was shown by example that a small {\em ad hoc} perturbation---preserving
the Hilbert function---of a monomial ideal without the weak Lefschetz  property may result in an ideal having the weak Lefschetz
property for almost every field characteristic.

Herein we propose a systematic way of deforming a monomial ideal without the weak Lefschetz  property to an ideal with the weak
Lefschetz property (in almost every characteristic) and the same Hilbert function as the original ideal. This could potentially be
useful, for example, if one expects an ideal to have a unimodal Hilbert function. Indeed, showing that the deformed ideal has the
weak Lefschetz property would then imply the desired unimodality.

The basic idea is to lift the monomial ideal to a finite set of points. We then expect the general hyperplane section of this set of
points to have the Lefschetz properties. We test this idea in the case of level monomial ideals in three variables of low type  that
do not have the weak Lefschetz property. If the type is one, then such an ideal is a complete intersection, so it has the weak
Lefschetz property in characteristic zero. The latter is also true if the type is two by \cite[Theorem~7.17]{BMMNZ}. Thus, we focus on
a family of almost complete intersections of type three that do not have the weak Lefschetz property. Lifting such an ideal to a finite
set of points we get a level set of points in three-space of type three. We show that the general hyperplane section of this set has
the weak Lefschetz property in almost every characteristic, whereas a special hyperplane section never has the weak Lefschetz property
(see Corollary \ref{cor:general-has-WLP}). Notice that examples of level sets of points in $\PP^3$ of type three such that {\em every}
Artinian reduction fails the weak Lefschetz property have been constructed in \cite[Section~3]{Mi}.

This note is organised as follows. In Section~\ref{sec:liftings} we recall the method of lifting an Artinian monomial ideal to a set of
points and we introduce the  family of Artinian monomial ideals that we focus on. In Section~\ref{sec:wlp},  we use this family to
explore the subtlety of the weak Lefschetz property under various hyperplane sections and in arbitrary characteristic.
Finally, in Section~\ref{sec:slp} we comment on the strong Lefschetz property in characteristic zero.

\section{Liftings and hyperplanes sections} \label{sec:liftings}

Let $R = K[x_0, \ldots, x_n]$ and $S = K[x_1, \ldots, x_n]$ be standard graded polynomial rings over an infinite field $K$.  Let $J
\subset S$ be a homogeneous ideal and $I \subset R$ be a homogeneous radical ideal.  Then we say that $J$ {\em lifts} to $I$ if $x_0$ is
a non-zero-divisor of $R/I$ and $(I,x_0)/(x_0) \cong J$. If such an $I$ exists, then $J$ is called a {\em liftable} ideal.

If $\alpha = (a_1, \ldots, a_n) \in \NN_0^n$, then define $x^\alpha := x_1^{a_1} \cdots x_n^{a_n}$; the degree of $x^\alpha$ is
$|\alpha| = \sum_{i=1}^n{a_i}$. For each $1 \leq j \leq n$, choose an infinite set $\{t_{j0}, t_{j1}, \ldots\} \subset K$ of distinct
elements.  Then to $\alpha \in \NN_0^n$ we associate the point $\overline{\alpha} := [1 : t_{1a_1} : \cdots : t_{na_n} ] \in
\PP^n_K$ and to $x^\alpha$ we associate the homogeneous polynomial 
\[
    \overline{x^\alpha} := \prod_{j=1}^n { \prod_{i=0}^{a_j-1} {(x_j - t_{ji}x_0)} } \in R.
\]

Using this construction, it was shown in~\cite[Theorem~4.9]{Ha} and~\cite[Theorem~2.2]{GGR} that monomial ideals are liftable given
that the field $K$ is infinite.

\begin{theorem} \label{thm:monomial-liftable}
    Let $I \subset S$ be a monomial ideal with minimal generators $\{x^{\alpha_1}, \ldots, x^{\alpha_m}\}$ and assume $K$ is infinite.
    Then $I$ lifts to the ideal $\oI := (\overline{x^{\alpha_1}}, \ldots, \overline{x^{\alpha_m}}) \subset R.$
\end{theorem}

Thus, given an Artinian monomial ideal $I \subset S$, the lifted ideal $\oI$ is a saturated ideal of a reduced set of points.  Moreover,
using~\cite[Proposition~2.6]{MN} we get that $\oI$ is level if and only if $I$ is level.  Hence, starting with a level Artinian monomial
ideal, the action of lifting yields a (Krull) dimension one saturated ideal of a reduced level set of points.

Let $Z$ be a subscheme of $\PP^n_K$ and let $H$ be a hyperplane (i.e., $\codim{H}=1$) in $\PP^n_K$.  Then $Z \cap H$ is a {\em hyperplane section}
of $Z$.  Given a linear form $h \in R$, we abuse notation and call $(\oI, h)$ a {\em hyperplane section} of $\oI$.  When $I$ is Artinian,
then $(\oI,h)$ is an Artinian reduction of $\oI$ if and only if $\dim{R/(\oI,h)} = 0$.  Specifically, $R/(\oI,x_0) \cong S/I$
is an Artinian reduction of $\oI$, hence all Artinian reductions of $\oI$ have the same Hilbert function as $I$.

\subsection*{A family of almost complete intersections}
Let $S = K[x,y,z]$ be the standard graded polynomial ring in three variables over an infinite field $K$.  Then for $t \geq 1$, define $I_t$ to be the ideal
\[ I_t := (x^{t+1}, y^{t+1}, z^{t+1}, xyz) \subset S. \]

\begin{proposition} \label{pro:properties-of-It}
    For $t \geq 1$, the ideal $I_t \subset S$ defined above has the following properties:
    \begin{enumerate}
        \item $S/I_t$ is level and Artinian,
        \item The minimal free resolution of $S/I_t$ has the form
            \[
                0 \longrightarrow S^3(-3-2t) \longrightarrow \triplestack{S^3(-3-t)}{\oplus}{S^3(-2-2t)} \longrightarrow
                \triplestack{S(-3)}{\oplus}{S^3(-1-t)} \longrightarrow S \longrightarrow S/I_t \longrightarrow 0,
            \]
            in particular, $S/I_t$ has socle type $3$, and
        \item The Hilbert function of $S/I_t$ is given by
            \[
                h_{S/I_t}(d) = \left\{ \begin{array}{ll}
                    1         & \mbox{if } d = 0; \\
                    3d        & \mbox{if } 1 \leq d \leq t; \\
                    3(2t+1-d) & \mbox{if } t < d \leq 2t; \\
                    0         & \mbox{if } t > 2t. \end{array}
                \right.
            \]
    \end{enumerate}
\end{proposition}
\begin{proof}
    The first two statements follow immediately from~\cite[Proposition~6.1]{MMN} and the third statement follows from (ii).
\end{proof}

A member of this family, $I_2$, was discussed in~\cite[Example~3.1]{BK-2007} where it was used to answer negatively the question of whether every almost complete
intersection in $S$ has the weak Lefschetz property.  Motivated by this, a larger family containing the $I_t$ is discussed in~\cite[Corollary~7.3]{Br},
\cite[Sections~6 and~7]{MMN}, and~\cite{CN}.  We continue this exploration by considering hyperplane sections of a lift of $I_t$.

We consider the particular lift of $I_t$ in $R = K[w,x,y,z]$ given by $t_{xi} = t_{yi} = t_{zi} = i$ for $0 \leq i \leq t$, that is, the homogeneous ideal
\[
    \oIt := \left(\prod_{i=0}^{t}{(x-iw)}, \prod_{i=0}^{t}{(y-iw)}, \prod_{i=0}^{t}{(z-iw)}, xyz\right) \subset R.
\]

If $2 \leq \charf{K} \leq t$, then $\oIt$ is not a {\em true} lifting of $I_t$, but we will consider it nonetheless.
When $\oIt$ is a true lifting, i.e., $\charf{K} = 0$ or $\charf{K} > t$, then $\oIt$ is the ideal of the level set of points
\[
    \{ [1 : a : b : c] \st 0 \leq a,b,c \leq t \mbox{ and at least one of $a,b,c$ is zero}\} \subset \PP^3_K,
\]
which is in bijection with the standard monomials of $S/I_t$.

Given the lift $\oIt$ of $I_t$, we consider the hyperplanes in $R$ of the form $w+ax$ for $a \in K$.  If $a \in N := \{y \st yi = -1 \mbox{ for some } i \in \{1, 2, \ldots, t\}\},$
then $w+ax$ is a zero-divisor of $R/\oIt$ and so $(\oIt, w+ax)$ is non-Artinian. Suppose $a \not\in N$, then $w+ax$ is a non-zero-divisor of $R/\oIt$ and so the
hyperplane section $(\oIt, w+ax)$ is Artinian.   Further still, $R/(\oIt, w+ax) \cong S/J_{t,a}$ where
\begin{equation} \label{equ:jta}
   J_{t,a} := \left(x^{t+1}, \prod_{i=0}^{t}{(iax + y)}, \prod_{i=0}^{t}{(iax + z)}, xyz\right) \subset S.
\end{equation}
Specifically, $J_{t,0} = I_t$.

We will next analyse the ideals $J_{t,a}$ for the presence of the Lefschetz properties.

\section{The weak Lefschetz property} \label{sec:wlp}

In~\cite[Proposition~2.2]{MMN}, it was shown that $x+y+z$ (and through a similar argument, $x+y-z$) is a weak Lefschetz element of
an Artinian monomial algebra if and only if the algebra has the weak Lefschetz property.  However, $S/J_{t,a}$ is not a monomial algebra
unless $a=0$.  We investigate whether the linear form $\ell := bx+cy-z$ is a weak Lefschetz  element of $S/J_{t,a}$.

Notice that $S/(J_{t,a},\ell) \cong T/L_{t,a}$ where $T = K[x,y]$ and
\begin{equation} \label{equ:lta}
    L_{t,a} := \left(x^{t+1}, \prod_{i=0}^{t}{(iax+y)}, \prod_{i=0}^{t}{((ia+b)x+cy)}, xy(bx+cy) \right) \subset T.
\end{equation}
The second and third generators of $L_{t,a}$ are products of linearly-consecutive binomial terms and can be described using the unsigned Stirling numbers
of the first kind.

The {\em unsigned Stirling numbers of the first kind}, denoted $s_{n,k}$, are defined recursively for $1 \leq k \leq n$ as $s_{n+1,k} = s_{n,k-1} + ns_{n,k}$
with the initial conditions $s_{1,1} = 1$ and $s_{n,0} = 0$ for $n \geq 1$.  In particular, $\prod_{i=0}^{n-1}(x+i) = \sum_{k=1}^n{s_{n,k}x^k}$ (see,
e.g., \cite[Theorem~1.3.4]{St-EC}).

We take the convention that both the empty product and $0^0$ are one.
\begin{lemma} \label{lem:stirling-a}
    Let $0 \neq a, b$ be in $K$ and let $k, n$ be integers with $1 \leq k \leq n$.  Define $V$ to be the set
    of $n$ elements of $K$ of the form $ia+b$ for $0 \leq i < n$.  Then the sum of all products of $k$ subsets in $V$ is
    \[
        d_{n,n-k}(a,b) := \sum_{i=0}^k{s_{n,n-i}\binom{n-i}{k-i}a^ib^{k-i}}.
    \]
\end{lemma}
\begin{proof}
    Clearly $d_{1,1}(a,b) = 1$, as $\emptyset$ is the unique subset of size zero.  Further, for all $m \geq 1$,
    \begin{eqnarray*}
        d_{m,0}(a,b) & = & \Pi_{i=0}^{m-1}(ia+b) \\
                     & = & \Sigma_{k=1}^m{s_{m,k}a^{m-k}b^{k}} \\
                     & = & \Sigma_{i=0}^m{s_{m,m-i}\tbinom{m-i}{m-i}a^ib^{k-i}}.
    \end{eqnarray*}

    Let $U$ be the set of $n$ elements of $K$ of the form $ia+b$ for $0 \leq i < n$ and $V = U \cup \{na+b\}$.  Then the sum of all products of $k$ subsets of
    $V$ is the sum of all products of $k$ subsets of $U$ plus the sum of all products of $k-1$ subsets of $U$ scaled by $na+b$.  That is,
    \begin{equation} \label{eqn:stirling-a-rec}
        d_{n+1,(n+1)-k}(a,b) = d_{n,n-k}(a,b) + (na+b)d_{n,n-(k-1)}(a,b).
    \end{equation}

    To proceed by induction, assume the equation on $d_{n,n-k}(a,b)$ is true for all $1 \leq k \leq n$.  Consider then $d_{n+1,(n+1)-k}(a,b)$ with
    $1 \leq k \leq n$.  By Equation~(\ref{eqn:stirling-a-rec}), we have
    \begin{eqnarray*}
        d_{n+1,(n+1)-k}(a,b) & = & d_{n,n-k}(a,b) + (na+b)d_{n,n-(k-1)}(a,b) \\
                             & = & \Sigma_{i=0}^k{s_{n,n-i}\tbinom{n-i}{k-i}a^ib^{k-i}} + (na+b)\Sigma_{i=0}^{k-1}{s_{n,n-i}\tbinom{n-i}{k-1-i}a^ib^{k-1-i}} \\
                             & = & \Sigma_{i=0}^k{s_{n,n-i}\tbinom{n+1-i}{k-i}a^ib^{k-i}} + \Sigma_{i=1}^k{ns_{n,n-(i-1)}\tbinom{n+1-i}{k-i}a^ib^{k-i}} \\
                             & = & \Sigma_{i=0}^k{(s_{n,n-i} + ns_{n,n-(i-1)})\tbinom{n+1-i}{k-i}a^ib^{k-i}} \\
                             & = & \Sigma_{i=0}^k{s_{n+1,(n+1)-i}\tbinom{n+1-i}{k-i}a^ib^{k-i}},
    \end{eqnarray*}
    where we use the properties of binomial coefficients and the unsigned Stirling numbers of the first kind as needed.
\end{proof}

It should be noted that $d_{n, n-k}(1,0) = s_{n,n-k}$ as the set $V$ is then $\{0, 1, \ldots, n-1\}$.

Using Lemma~\ref{lem:stirling-a}, the second generator of $L_{t,a}$ can be described using the unsigned Stirling numbers of the first kind as
\[
    \prod_{i=0}^{t}{(iax+y)} = \sum_{i=0}^{t}s_{t+1,t+1-i}a^{i}x^{i}y^{t+1-i} = \sum_{i=0}^{t}d_{t+1,t+1-i}(a,0)x^iy^{t+1-i}
\]
and the third generator of $L_{t,a}$ can be described using Lemma~\ref{lem:stirling-a} as
\[
    \prod_{i=0}^{t}{((ia+b)x+cy)} = \sum_{i=0}^{t}d_{t+1,t+1-i}(a,b)c^{t+1-i}x^{i}y^{t+1-i}.
\]

Now we return to studying the weak Lefschetz property.  We have explicitly described the coefficients of the generators of
$L_{t,a}$, hence we can use linear algebra to determine whether $S/J_{t,a}$ has $\ell = bx+cy-z$ as a weak Lefschetz element. Define
$N = \{y \st yi = -1 \mbox{ for some } i \in \{1, 2, \ldots, t\}\}$, as above.
\begin{proposition} \label{pro:matrix-Lta}
    Consider the algebra $A = S/J_{t,a}$ as in Equation~(\ref{equ:jta}) and $a \not\in N$.  Let $M_{t,a,b,c}$ be the $(t+1)\times(t+1)$ $K$-matrix given by
    \[
        \left[
        \begin{array}{ccccccc}
            s_{t+1,1}a^t    & s_{t+1,2}a^{t-1}  & s_{t+1,3}a^{t-2}  & \cdots & s_{t+1,t-1}a^2          & s_{t+1,t}a          & 1 \\
            d_{t+1,1}(a,b)c & d_{t+1,2}(a,b)c^2 & d_{t+1,3}(a,b)c^3 & \cdots & d_{t+1,t-1}(a,b)c^{t-1} & d_{t+1,t}(a,b)c^{t} & c^{t+1} \\
            b               & c                 & 0                 & \cdots & 0                       & 0                   & 0 \\
            0               & b                 & c                 & \cdots & 0                       & 0                   & 0 \\
                            &                   &                   & \vdots &                         &                     &   \\
            0               & 0                 & 0                 & \cdots & b                       & c                   & 0
        \end{array}
        \right].
    \]
    Then the algebra $A = S/J_{t,a}$, from Equation~(\ref{equ:jta}), has $\ell = bx+cy-z$ as a weak Lefschetz element if and only if $\det{M_{t,a,b,c}}$ is nonzero in $K$.

    Thus, $\det{M_{t,a,b,c}} \neq 0 \in K$ for some $b, c \in K$ if and only if $A$ has the weak Lefschetz property.
\end{proposition}
\begin{proof}
    Since $J_{t,a}$ is an Artinian reduction of the lift of $I_t$, then their Hilbert functions are equal.  Thus using Proposition~\ref{pro:properties-of-It}(iii), we
    have then that the Hilbert function of $S/J_{t,a}$ is strictly unimodal from $0$ to $2t$ and has a twin-peak at $t$ and $t+1$; that is,
    $h_{S/J_{t,a}}(t) = 3t = h_{S/J_{t,a}}(t+1)$.  Hence by~\cite[Proposition~2.1]{MMN}, $S/J_{t,a}$ has the weak Lefschetz property if and only if the map
    $[S/J_{t,a}]_t \stackrel{bx+cy-z}{\longrightarrow} [S/J_{t,a}]_{t+1}$ is an isomorphism for some $b,c\in K$ and thus it suffices to check whether
    $[T]_{t+1} \subset L_{t,a}$.

    Thus, the matrix $M_{t,a,b,c}$ corresponds to the system of equations which needs to be solved to determine if a polynomial in $[T]_{t+1}$ with no $x^{t+1}$ term
    is in $L_{t,a}$.  Hence, $\det{M_{t,a,b,c}} \neq 0 \in K$ if and only if $M_{t,a,b,c}$ is invertible in $K$, i.e., $[T]_{t+1} \subset L_{t,a}$.
\end{proof}

Furthermore, we give a closed form for the determinant of $M_{t,a,b,c}$.
\begin{proposition} \label{pro:det-Mta}
    Assuming $b,c$ both nonzero, then the determinant of $M_{t,a,b,c}$ is
    \[
        (-1)^tc^t\left( \prod_{i=1}^{t}(ai+b) - \prod_{i=1}^{t}(aci-b) \right).
    \]
\end{proposition}
\begin{proof}
    By straightforward Gaussian elimination, we get
    \begin{eqnarray*}
        \det{M_{t,a,b,c}} & = & b^{t-1}c^t\sum_{j=1}^t\left( \left(-\frac{c}{b}\right)^{t-j}ca^{t+1-j}s_{t+1,j} - \left(-\frac{1}{b}\right)^{t-j}d_{t+1,j}(a,b) \right) \\
                          & = & b^{t-1}c^t\left( \sum_{j=1}^{t}\left(-\frac{c}{b}\right)^{t-j}ca^{t+1-j}s_{t+1,j} \right.\\
                          &   & \left.  - \sum_{k=1}^{t+1}a^{t+1-k}s_{t+1,k}\sum_{j=1}^{t}\left(-\frac{1}{b}\right)^{t-j}b^{k-j}\binom{k}{k-j} \right) \\
                          & = & (-1)^tc^t\sum_{j=1}^{t+1}a^{t+1-j}b^{j-1}s_{t+1,j}\left( (-1)^jc^{t+1-j}+1 \right) \\
                          & = & (-1)^tc^t\left( \sum_{j=1}^{t+1}a^{t+1-j}b^{j-1}s_{t+1,j} + \sum_{j=1}^{t+1}(-1)^j a^{t+1-j}b^{j-1}c^{t+1-j}s_{t+1,j} \right) \\
                          & = & (-1)^tc^t\left( \prod_{i=1}^{t}(ai+b) - \prod_{i=1}^{t}(aci-b) \right).
    \end{eqnarray*}
\end{proof}

Given the above determinant calculation, we make the following
observations.
\begin{remark}\label{rem:specific-dets}
    If we specialise the parameters $a, b,$ and $c$ suitably, then we get three nice determinant evaluations.
    \begin{enumerate}
        \item {\em When $a=0$}:  Then $J_{t,0} = I_t$ and $\det{M_{t,0,b,c}} = b^tc^t((-1)^t-1)$.  Hence $S/J_{t,0}$ has the weak Lefschetz property
            if and only if $t$ is odd and $\charf{K} \neq 2$, recovering~\cite[Proposition~3.1]{CN}.
        \item {\em When $b = c = 1$}:  Then
            \[
                \det{M_{t,a,1,1}} =
                \left\{ \begin{array}{ll}
                        2\sum_{i=1}^{t/2}a^{2i-1}s_{t+1,t+1-(2i-1)} & \mbox{if $t$ is even;} \\
                        -2(1+\sum_{i=1}^{\floor{t/2}}a^{2i}s_{t+1,t+1-2i}) & \mbox{if $t$ is odd.}
                \end{array} \right.
            \]
            Thus in characteristic zero, $x+y-z$ is a weak Lefschetz  element of $S/J_{t,a}$ if $a \neq 0$ and $a \not\in N$.
        \item {\em When $a = b = c = 1$}: Then $\det{M_{t,1,1,1}} = (-1)^t(t+1)!$.  Hence $x+y-z$ is a weak Lefschetz element of $S/J_{t,1}$ if and only if $\charf{K} = 0$ or
            $\charf{K} > t$, i.e., $\oIt$ is a true lifting of $I_t$.
    \end{enumerate}
\end{remark}

\begin{theorem} \label{thm:Jta-wlp}
    Let $K$ be any infinite field and $S = K[x,y,z]$.  Then for all $a$ in $K$ such that $a \neq 0$ and $a\not\in N$ and for all positive integers $t$, the algebra
    $A = S/J_{t,a}$ has the weak Lefschetz property.
\end{theorem}
\begin{proof}
    Let $a, c$ be nonzero elements of $K$.  Then $b = ac$ is nonzero and moreover
    \[
        \det{M_{t,a,b,c}} = (-1)^tc^t\prod_{i=1}^{t}(ai+b) = (-1)^ta^tc^t\prod_{i=1}^{t}(i+c) \subset K[x, y, z].
    \]
    As $K$ is infinite, there exists a nonzero $c$ in $K$ such that $i+c \neq 0$ for $1 \leq i \leq t$.  Hence,
    $\det{M_{t,a,b,c}}$ is nonzero in $K$.  Therefore, if $a \not\in N$, then $A$ has the weak Lefschetz property.
\end{proof}

We partially summarise our results as follows:

\begin{corollary} \label{cor:general-has-WLP}
    Let $t \ge 1$ be an integer and set $A = R/\oIt$, where
    \[
        \oIt := \left(\prod_{i=0}^{t}{(x-iw)}, \prod_{i=0}^{t}{(y-iw)}, \prod_{i=0}^{t}{(z-iw)}, xyz\right) \subset R = K[w,x,y,z].
    \]
    Then:
    \begin{enumerate}
        \item If the characteristic of $K$ is zero or  greater than $t$, then the ideal $\oIt$ defines a set of $3 (t+1) t +1$  points
            in $\PP^3$ that is level of type three.
        \item If $\ell \in [R]_1$ is a general linear form, then $A/\ell A$ has the weak Lefschetz property.
        \item If $\ell = w$, then the Artinian algebra $A/\ell A$ has the weak Lefschetz property if and only if $t$ is odd and $\charf K \neq 2$.
    \end{enumerate}
\end{corollary}

\begin{proof}
    Claim (i) follows by Theorem \ref{thm:monomial-liftable} and Proposition \ref{pro:properties-of-It}.

    In order to prove (ii), notice that Theorem \ref{thm:Jta-wlp} shows that $A/\ell' A$ has the weak Lefschetz property, where $\ell' = w +
    a x$ and $a \neq 0$ is any element in $K \setminus N$. Let $L \in R$ be
    another general linear form. Then, for all $j \in \ZZ$, one has
    \[
        \dim_K [A/(\ell, L)A]_j \le \dim_K [A/(\ell', L)A]_j
    \]
    Since $A/\ell'A$ has the weak Lefschetz property, this must be an equality, hence $A/\ell A$ also has the weak Lefschetz property.

    Part (iii) has been shown in Remark \ref{rem:specific-dets}(i).
\end{proof}

Specialising to $t = 2$, we get an example reminiscent of~\cite[Example~3.1]{BK-2007}, which showed that, in characteristic zero, for any degree three form $f$, the ideal
$(x^3, y^3, z^3, f)$ has the weak Lefschetz property if and only if $f \neq xyz$ modulo $x^3, y^3, z^3$.
\begin{example} \label{exa:t-2}
    Let $t = 2$, $\charf{K} \neq 2$, and $a \in K \setminus \{-1, -\frac{1}{2}\}$, then
    \[
        J_{t,a} = (x^3, y(ax+y)(2ax+y), z(ax+z)(2ax+z), xyz)
    \]
    is Artinian.  Moreover, $S/J_{t,a}$ has the weak Lefschetz property if and only if $a \neq 0$, that is, $J_{t,a}$ is non-monomial.
\end{example}

\section{The strong Lefschetz property} \label{sec:slp}

In case, $a = 0$, where $S/J_{t,a}$ is a monomial algebra, the strong Lefschetz property fails spectacularly.
\begin{proposition} \label{pro:slp-a-0}
    The algebra $S/J_{t,0} = S/I_t$ has the strong Lefschetz property if and only if $t = 1$ and $\charf{K} \neq 2$.
\end{proposition}
\begin{proof}
    Let $\ell = x+y-z$, then by~\cite[Proposition~2.2]{MMN}, $\ell$ is a strong Lefschetz element of $S/J_{t,0}$ if and only if $S/J_{t,0}$ has the strong
    Lefschetz property.

    If $t$ is even or $\charf{K} = 2$ then, by Proposition~\ref{pro:matrix-Lta}, $S/J_{t,0}$ fails to have the weak Lefschetz property hence fails to have the
    strong Lefschetz property.

    Suppose then $t$ is odd and $\charf{K} \neq 2$.  If $t = 1$, then, by Proposition~\ref{pro:matrix-Lta}, $S/J_{1,0}$ has the weak Lefschetz property.
    As the regularity of $S/J_{1,0}$ is two and the map $[S/J_{1,0}]_0 \stackrel{\ell^2}{\longrightarrow} [S/J_{1,0}]_2$ is injective since
    $\ell^2 \not\in J_{1,0}$, then $S/J_{1,0}$ has the strong Lefschetz property.

    Suppose $t\geq3$ and let $A = S/J_{t,0}$.  As $\dim_K{[A]_2} = 6 = \dim_K{[A]_{2t-1}}$ by Proposition~\ref{pro:properties-of-It}(iii), then it suffices to show
    the map $\varphi: [A]_{2} \stackrel{\ell^{2t-3}}{\longrightarrow} [A]_{2t-1}$ is not injective.
    Let $p,q \in K$, not both zero, such that $p(2t-2)+qt=0$; such a non-trivial solution exists in $K$ regardless of characteristic.  Consider then
    $f = p(x^2+y^2+z^2)+q(xy+xz+yz)$ which is a nonzero element of $[A]_2$.  As $\ell^{2t-3}f$ is equivalent to
    $(p(2t-2)+qt)(x^ty^{t-1}+x^{t-1}y^t+x^tz^{t-1}-x^{t-1}z^t+x^tz^{t-1}-x^{t-1}z^t)$ modulo $I_t$, then $\ell^{2t-3}f$ is in $I_t$.  Hence, $\varphi$
    is not injective and thus $A$ fails to have the strong Lefschetz property.
\end{proof}

Now we consider the case when the Artinian algebra  $S/J_{t,a}$ is
{\em not} a monomial algebra.
\begin{remark} \label{rem:slp-a-nz}
    Suppose $K$ is a field of characteristic zero.  Let $a \not\in N$ be a nonzero element of $K$ and let $A = S/J_{t,a}$.  As the Hilbert function of
    $A$ is symmetric from $1$ to $2t$ with peak $t,t+1$ and $A$ is level by Proposition~\ref{pro:properties-of-It}, then using~\cite[Proposition~2.1]{MMN} it suffices to show
    for $1 \leq i \leq t$ the following hold:
    \begin{enumerate}
        \item $[A]_{t-i+1} \stackrel{\ell^{2i-1}}{\longrightarrow} [A]_{t+i}$ is an isomorphism,
        \item $[A]_{t-i} \stackrel{\ell^{2i}}{\longrightarrow} [A]_{t+i}$ is an injection, and
        \item $[A]_{t-i+1} \stackrel{\ell^{2i}}{\longrightarrow} [A]_{t+i+1}$ is a surjection.
    \end{enumerate}

    As $A$ has the weak Lefschetz property, if part~(i) is shown for all $i$, then parts~(ii) and~(iii) follow immediately:
    \begin{enumerate}
        \item[(ii)] $[A]_{t-i} \stackrel{\ell^{2i}}{\rightarrow} [A]_{t+i} = [A]_{t-i} \stackrel{\ell}{\rightarrow} [A]_{t-i+1} \stackrel{\ell^{2i-1}}{\rightarrow} [A]_{t+i}$
            is a composition of injective maps, hence injective.
        \item[(iii)] $[A]_{t-i+1} \stackrel{\ell^{2i}}{\rightarrow} [A]_{t+i+1} = [A]_{t-i+1} \stackrel{\ell^{2i-1}}{\rightarrow} [A]_{t+i} \stackrel{\ell}{\rightarrow} [A]_{t+i+1}$
            is a composition of surjective maps, hence surjective.
    \end{enumerate}

    Thus, in order to show $A$ has the strong Lefschetz property it is sufficient to show part~(i) holds for $2 \leq i \leq t$.
\end{remark}

Using~\cite{M2}, we have verified that in characteristic zero, if $t \le 30$, then there exists some $a \in K$ such that $\ell$ is a
strong Lefschetz element of $S/J_{t,a}$.  We suspect that this is always true.
\begin{conjecture} \label{con:slp-a-nz}
    Suppose $K$ is a field of characteristic zero.  Let $a \not\in N$ be a nonzero element of $K$ and let $A = S/J_{t,a}$.
    Then $A$ has the strong Lefschetz property with strong Lefschetz element $\ell = x+y-z$.
\end{conjecture}

Thus, if the conjecture holds, then, at least in characteristic zero, there is only one ``bad'' choice for the strong Lefschetz property and it is,
interesting in its own right, the only monomial case.


\end{document}